\documentclass[reqno,11pt]{amsart}
\usepackage{amssymb,amsmath,amsthm,amsxtra,mathtools,mathrsfs,calc,bm}
\usepackage{enumerate}
\usepackage[margin=1in]{geometry}
\usepackage{nccmath}

\def\Z{\mathbb{Z}}
\def\Q{\mathbb{Q}}

\def\R{\mathbb{R}}
\def\H{\mathbb{H}}
\def\N{\mathbb{N}}
\def\C{\mathbb{C}}
\def\P{\mathbb{P}}

\def\SL{{\rm SL}}

\def\Mp{{\rm Mp}}

\newcommand{\pmmatrix}[4]{\arraycolsep=.3em\def\arraystretch{1.1}\Bigl( \, \begin{mmatrix} #1&#2\\#3&#4\end{mmatrix} \, \Bigr)}
\renewcommand{\pmatrix}[4]{\left(\begin{smallmatrix}#1 & #2 \\ #3 & #4\end{smallmatrix}\right)}
\renewcommand{\bar}[1]{\overline{#1}}

\renewcommand{\sl}{\big|}

\DeclareMathOperator{\ord}{ord}

\makeatletter
\def\mathcenterto#1#2{\mathclap{\phantom{#1}\mathclap{#2}}\phantom{#1}}
\let\old@widetilde\widetilde
\def\widetildeto#1#2{\mathcenterto{#2}{\old@widetilde{\mathcenterto{#1}{#2\,}}}}
\makeatother

\def\til{\widetildeto{f}}

\newtheorem{theorem}{Theorem}
\newtheorem{lemma}[theorem]{Lemma}

\newtheorem{proposition}[theorem]{Proposition}

\theoremstyle{remark}

\numberwithin{equation}{section}

\renewcommand*{\arraystretch}{1.3}

\begin{document}

\title[Mock theta conjectures]{Vector-valued modular forms and \\ the Mock Theta Conjectures}
\date{\today}
\author{Nickolas Andersen}

\begin{abstract}
	The mock theta conjectures are ten identities involving Ramanujan's fifth-order mock theta functions.
	The conjectures were proven by Hickerson in 1988 using $q$-series methods.
	Using methods from the theory of harmonic Maass forms, specifically work of Zwegers and Bringmann--Ono, Folsom reduced the proof of the mock theta conjectures to a finite computation.
	Both of these approaches involve proving the identities individually, relying on work of Andrews--Garvan.
	Here we give a unified proof of the mock theta conjectures by realizing them as an equality between two nonholomorphic vector-valued modular forms which transform according to the Weil representation.
	We then show that the difference of these vectors lies in a zero-dimensional vector space.
\end{abstract}

\maketitle

\allowdisplaybreaks

\section{Introduction}

In his last letter to Hardy, dated three months before his death in early 1920, Ramanujan briefly described a new class of functions which he called mock theta functions, and he listed 17 examples \cite[p.~220]{ramanujan-letters}.
These he separated into three groups: four of third order, ten of fifth order, and three of seventh order.
The fifth order mock theta functions he further divided into two groups\footnote{While Ramanujan reused the letters $f$, $\phi$, $\psi$, $\chi$, and $F$ in each group, the usual convention is to write those in the first group with a subscript `0' and those in the second group with a subscript `1'.}; for example, four of these fifth order functions are
\begin{alignat*}{2}
	f_0(q) &= \sum_{n=0}^\infty \frac{q^{n^2}}{(-q;q)_n}, \quad & \quad f_1(q) &= \sum_{n=0}^\infty \frac{q^{n(n+1)}}{(-q;q)_n}, \\
	F_0(q) &= \sum_{n=0}^\infty \frac{q^{2n^2}}{(q;q^2)_n}, & F_1(q) &= \sum_{n=0}^\infty \frac{q^{2n(n+1)}}{(q;q^2)_{n+1}}. 
\end{alignat*}
Here we have used the standard $q$-Pochhammer notation $(a;q)_n := \prod_{m=0}^{n-1} (1-aq^m)$.

The mock theta conjectures are ten identities found in Ramanujan's lost notebook, each involving one of the fifth-order mock theta functions.
The identities for the four mock theta functions listed above are (following the notation of \cite[p.~206]{gordon-mcintosh-57}, and correcting a sign error in the fourth identity in that paper; see also \cite{andrews-garvan,gordon-mcintosh-survey})
\begin{align}
	f_0(q) &= -2q^2 M\left(\mfrac 15, q^{10}\right) + \theta_4(0,q^5) G(q), \label{mtc-1} \\
	f_1(q) &= -2q^3 M\left(\mfrac 25, q^{10}\right) + \theta_4(0,q^5) H(q), \label{mtc-2}\\
	F_0(q) - 1 &= q M\left(\mfrac 15, q^{5}\right) - q \psi(q^5)H(q^2), \label{mtc-3}\\
	F_1(q) &= q M\left(\mfrac 25, q^{5}\right) + \psi(q^5)G(q^2). \label{mtc-4}
\end{align}
Here
\begin{equation*}
	M(r,q) := \sum_{n=1}^\infty \frac{q^{n(n-1)}}{(q^r;q)_n(q^{1-r};q)_n},
\end{equation*}
the functions $\theta_4(0,q)$ and $\psi(q)$ are theta functions, and $G(q)$ and $H(q)$ are the Rogers-Ramanujan functions (see Section~\ref{sec:preliminaries} for definitions).
Andrews and Garvan \cite{andrews-garvan} showed that the mock theta conjectures fall naturally into two families of five identities each 
(according to Ramanujan's original grouping), and that within each family the truth of each of the identities implies the truth of the others via straightforward $q$-series manipulations.
Shortly thereafter Hickerson \cite{hickerson-proof} proved the mock theta conjectures by establishing the identities involving $f_0(q)$ and $f_1(q)$.
According to Gordon and McIntosh \cite[p.~106]{gordon-mcintosh-survey}, the mock theta conjectures together form ``one of the fundamental results of the theory of [mock theta functions]'' and Hickerson's proof is a ``tour de force.''

In his PhD thesis \cite{zwegers}, Zwegers showed that the mock theta functions can be completed to real analytic modular forms of weight 1/2 by multiplying by a suitable rational power of $q$ and adding nonholomorphic integrals of certain unary theta series of weight 3/2.
This allows the mock theta functions to be studied using the theory of harmonic Maass forms.
Bringmann, Ono, and Rhoades remark in \cite[p.~1087]{bor-eulerian} that their Theorem 1.1, together with the work of Zwegers, reduces the proof of the mock theta conjectures to ``the verification of two simple identities for classical weakly holomorphic modular forms.''
Zagier makes a similar comment in \cite[\S 6]{zagier-survey}.
Following their approach, Folsom \cite{folsom-mock-theta} reduced the proof of the $\chi_0(q)$ and $\chi_1(q)$ mock theta conjectures to the verification of two identities in the space of modular forms of weight $1/2$ for the subgroup $G=\Gamma_1(144\cdot 10^2 \cdot 5^4)$.
Since $[\SL_2(\Z):G]\geq 5\times 10^{13}$, this computation is currently infeasible.

The purpose of the present paper is to provide a conceptual, unified proof of the mock theta conjectures that relies neither on computational verification nor on the work Andrews and Garvan \cite{andrews-garvan}.
Our method proves four of the ten mock theta conjectures simultaneously; two from each family (namely the identities \eqref{mtc-1}--\eqref{mtc-4} above).
Four of the remaining six conjectures can be proved using the same method, and the remaining two follow easily from the others (see Section~\ref{sec:remaining}).

To accomplish our goal, we recast the mock theta conjectures in terms of an equality between two nonholomorphic vector-valued modular forms $\bm F$ and $\bm G$ of weight 1/2 on $\SL_2(\Z)$ which transform according to the Weil representation (see Lemma~\ref{lem:weil-rep} below),
and we show that the difference $\bm F-\bm G$ is holomorphic.
Employing a natural isomorphism between the space of such forms and the space $J_{1,60}$ of Jacobi forms of weight $1$ and index $60$, together with the result of Skoruppa that $J_{1,m}=\{0\}$, we conclude that $\bm F=\bm G$.

\section{Definitions and transformations} \label{sec:preliminaries}

In this section, we define the functions $M(\frac a5,q)$, $\theta_4(0,q)$, $\psi(q)$, $G(q)$, and $H(q)$ and describe the transformation behavior for these functions and the mock theta functions under the generators
\[
	T := \pmmatrix 1101 \quad \text{ and } \quad S := \pmmatrix 0{-1}10
\]
of $\SL_2(\Z)$.
We employ the usual $\sl_k$ notation, defined for $k\in \frac 12 \Z$ and $\gamma=\pmatrix abcd\in \SL_2(\Z)$ by
\[
	(f \sl_k \gamma )(z) := (cz+d)^{-k} f\left(\mfrac{az+b}{cz+d}\right).
\]
For $k\in \Z$ we have $f\sl_k AB=f\sl_k A\sl_k B$, but for general $k$ we have
\begin{gather} \label{eq:sl}
	f\sl_k A B = e^{i k \left(\arg j_A(Bi)+\arg j_B(i)-\arg j_{AB}(i) \right)} f\sl_k A\sl_k B,
\end{gather}
where $j_\gamma(z)=cz+d$ for $\gamma=\pmatrix **cd$ (see \cite[\S 2.6]{iwaniec}).
When $k\notin \Z$, we always take $\arg z \in (-\pi,\pi]$.
Much of the arithmetic here and throughout the paper takes place in the splitting field of the polynomial $x^4-5x^2+5$, which has roots
\begin{equation} \label{eq:al-be}
	\alpha := \sqrt{\mfrac{5-\sqrt 5}{2}} \quad \text{ and } \quad \beta := \sqrt{\mfrac{5+\sqrt 5}{2}}.
\end{equation}

We begin by giving the modular transformations satisfied by the mock theta functions $f_0$, $f_1$, $F_0$, and $F_1$ which are given in Section 4.4 of \cite{zwegers}.
The nonholomorphic completions are given in terms of the integral (see \cite[Proposition 4.2]{zwegers})
\[
	R_{a,b}(z) := -i \int_{-\bar z}^{i\infty} \frac{g_{a,-b}(\tau)}{\sqrt{-i(\tau+z)}} d\tau,
\]
where $g_{a,b}$ (see \cite[\S 1.5]{zwegers}) is the unary theta function
\[
	g_{a,b}(z) := \sum_{\nu\in a+\Z} \nu e^{\pi i \nu^2 z + 2\pi i \nu b}.
\]
We will simplify the components of $G_{5,1}(\tau)$ on page 75 of \cite{zwegers} by using the relation
\begin{align*}
	g_{a,0}(z) - g_{a+\frac 12,0}(z) &= \mfrac 12 e^{-2\pi i a} g_{2a,\frac 12}(z/4)
\end{align*}
and Proposition 1.15 of \cite{zwegers}.
As usual, $q:=\exp(2\pi iz)$.
We define
\begin{align}
	\widetilde f_0(z) &:= q^{-\frac1{60}} f_0(q) - \zeta_{10} \, \left(\zeta_{12}^{-1} \, R_{\frac 1{30},\frac 12} + \zeta_{12} \, R_{\frac {11}{30},\frac 12}\right)(30z), \label{eq:mt-complete-1} \\
	\widetilde f_1(z) &:= q^{\frac{11}{60}} f_1(q) - \zeta_{5} \left(\zeta_{12}^{-1} \, R_{\frac {7}{30},\frac 12} + \zeta_{12} \, R_{\frac {17}{30},\frac 12}\right)(30z), \label{eq:mt-complete-2} \\
	\widetilde F_0(z) &:= q^{-\frac 1{120}} (F_0(q)-1) + \mfrac 12 \, \zeta_{10} \left( \zeta_{12}^{-1} \, R_{\frac{1}{30},\frac 12} + \zeta_{12} \, R_{\frac{11}{30},\frac 12} \right)(15z), \label{eq:mt-complete-3} \\
	\widetilde F_1(z) &:= q^{\frac {71}{120}} F_1(q) + \mfrac 12 \, \zeta_5 \left( \zeta_{12}^{-1} \, R_{\frac{7}{30},\frac 12} + \zeta_{12} R_{\frac{17}{30},\frac 12} \right)(15z). \label{eq:mt-complete-4}
\end{align}
The following is Proposition 4.10 of \cite{zwegers}.
The vector \eqref{eq:vec-F-def} below equals the vector $F_{5,1}(\tau)-G_{5,1}(\tau)$ of that paper (some computation is required to see this for the fifth and sixth components).

\begin{proposition} \label{prop:F-transformation}
	The vector
	\begin{equation} \label{eq:vec-F-def}
		\bm F(z) := \left( \ \til f_0(z), \ \til f_1(z), \ \til F_0\left(\tfrac z2\right), \ \til F_1\left(\tfrac z2\right), \ \zeta_{240} \, \til F_0\left(\tfrac {z+1}2\right), \ \zeta_{240}^{-71}\,\til F_1\left(\tfrac {z+1}2\right) \ \right)^{\intercal}
	\end{equation}
	satisfies the transformations
	\begin{equation*}
		\bm F \sl_{\frac 12} T = M_T \, \bm F \quad \text{ and } \quad \bm F \sl_{\frac 12} S = \zeta_8^{-1} \sqrt{\mfrac 25} \, M_S \, \bm F,
	\end{equation*}
	where
	\begin{equation*}
	\def\arraystretch{1.2}
	\arraycolsep=.2em
	M_T = 
	\left(
	\begin{array}{cccccc}
		\zeta_{60}^{-1} & 0 & 0 & 0 & 0 & 0 \\
		0 & \zeta_{60}^{11} & 0 & 0 & 0 & 0 \\
		0 & 0 & 0 & 0 & \zeta_{240}^{-1} & 0 \\
		0 & 0 & 0 & 0 & 0 & \zeta_{240}^{71} \\
		0 & 0 & \zeta_{240}^{-1} & 0 & 0 & 0 \\
		0 & 0 & 0 & \zeta_{240}^{71} & 0 & 0
	\end{array}
	\right)
	\quad \text{ and } \quad 
	\arraycolsep=.35em
	M_S =  \left(
	\begin{array}{cccccc}
		0 & 0 & \alpha & \beta & 0 & 0 \\
		0 & 0 & \beta & -\alpha & 0 & 0 \\
		\frac \alpha2  & \frac \beta2 & 0 & 0 & 0 & 0 \\
		\frac \beta2   & -\frac \alpha2 & 0 & 0 & 0 & 0 \\
		0 & 0 & 0 & 0 & \frac \beta{\sqrt 2} & \frac \alpha{\sqrt 2} \\
		0 & 0 & 0 & 0 & \frac \alpha{\sqrt 2} & -\frac \beta{\sqrt 2}
	\end{array}
	\right).
	\end{equation*}
\end{proposition}

Next we define the functions on the right-hand side of \eqref{mtc-1}--\eqref{mtc-4} and give their transformation properties.
Following \cite{bringmann-ono-dyson,gordon-mcintosh-57}, we define, for $a\in \{1,2,3,4\}$, the functions
\begin{align}
	M\left(\mfrac a5,z\right) &:= \sum_{n=1}^\infty \frac{q^{n(n-1)}}{(q^{\frac a5};q)_n(q^{1-\frac a5};q)_n},\\
	N\left(\mfrac a5,z\right) &:= 1 + \sum_{n=1}^\infty \frac{q^{n^2}}{(\zeta_5^a q;q)_n(\zeta_5^{-a}q;q)_n}. \label{eq:N-def}
\end{align}
Clearly we have $M(1-\frac a5,z)=M(\frac a5,z)$ and $N(1-\frac a5,z)=N(\frac a5,z)$.
Bringmann and Ono \cite{bringmann-ono-dyson} also define auxiliary functions $M(a,b,5,z)$ and $N(a,b,5,z)$ for $0\leq a\leq 4$ and $1\leq b \leq 4$.
Together, the completed versions of these functions form a set that is closed (up to multiplication by roots of unity) under the action of $\SL_2(\Z)$ (see \cite[Theorem 3.4]{bringmann-ono-dyson}).
Garvan \cite{garvan} made the definitions of these functions and their transformations more explicit, so in what follows we reference his paper.

The nonholomorphic completions for $M(\frac a5,z)$ and $N(\frac a5,z)$ are given in terms of integrals of weight 3/2 theta functions $\Theta_1(\frac a5,z)$ and $\Theta_1(0,-a,5,z)$ (defined in Section 2 of \cite{garvan}).
A straightforward computation shows that
\begin{align*}
	\Theta_1(0,-a,5,z) = 15\sqrt{3} \, \zeta_{10}^a \left( \zeta_{12}^{-1} \, g_{\frac{6a-5}{30},-\frac 12}(3z) + \zeta_{12} \, g_{\frac{6a+5}{30},-\frac 12}(3z) \right).
\end{align*}
Following (2.1), (2.2), (3.5), and (3.6) of \cite{garvan}, we define
\begin{align}
	\til M\left(\mfrac a5,z\right) &:= 2q^{\frac{3a}{10}\left(1-\frac a5\right)-\frac 1{24}} \, M\left(\mfrac a5, z\right) + \zeta_{10}^{a} \left( \zeta_{12}^{-1} R_{\frac{6a-5}{30},\frac 12} + \zeta_{12} R_{\frac{6a+5}{30},\frac 12} \right)(3z), \label{eq:mtil-def} \\
	\til N\left(\mfrac a5,z\right) &:= \csc \left(\mfrac {a\pi}5\right) q^{-\frac1{24}} N\left(\mfrac a5, z\right) 
	+\mfrac i{\sqrt 3} \int_{-\bar z}^{i\infty} \frac{\Theta_1(\frac a5,\tau)}{\sqrt{-i(\tau+z)}} \, d\tau. 
	\label{eq:ntil-def}
\end{align}
The completed functions $\til M(a,b,z):=\mathcal G_2(a,b,5;z)$ and $\til N(a,b,z):=\mathcal G_1(a,b,5;z)$ are defined in (3.7) and (3.8) of that paper. 
By Theorems 3.1 and 3.2 of \cite{garvan} we have
\begin{align}
	\til M\left(\mfrac a5,z\right) \sl_{\frac 12} T^5 &= \til M\left(\mfrac a5,z\right) \times 
	\begin{cases}
		\zeta_{120}^{-1} & \text{ if }a=1, \\
		\zeta_{120}^{71} & \text{ if }a=2,
	\end{cases} \label{eq:M-T}\\
	\til N\left(\mfrac a5,z\right) \sl_{\frac 12} T &= \zeta_{24}^{-1}\til N\left(\mfrac a5,z\right),
\end{align}
and
\begin{align}	\label{eq:M-S}
	\til M\left(\mfrac a5,z\right) \sl_{\frac 12} S = \zeta_8^{-1} \til N\left(\mfrac a5,z\right).
\end{align}

The theta functions $\theta_4(0,q)$ and $\psi(q)$ are defined by
\begin{align*}
	\theta_4(0,q) &:= \frac{(q;q)_\infty^2}{(q^2;q^2)_\infty} = \frac{\eta^2(z)}{\eta(2z)} \\
	\psi(q) &:= \frac{(q^2;q^2)^2_\infty}{(q;q)_\infty} = q^{-\frac 18}\frac{\eta^2(2z)}{\eta(z)},
\end{align*}
where $\eta(z) = q^{1/24}(q;q)_\infty$ is the Dedekind eta function.
The transformation properties of these functions are easily obtained using the well-known transformation 
\begin{equation} \label{eq:eta-S}
	\eta(-1/z) = \sqrt{-i z} \, \eta(z).
\end{equation}

The Rogers-Ramanujan functions are defined by
\begin{align*}
	G(q) &:= \frac{1}{(q;q^5)_\infty (q^4;q^5)_\infty}, \\
	H(q) &:= \frac{1}{(q^2;q^5)_\infty (q^3;q^5)_\infty}.
\end{align*}
It will be more convenient for us to use the functions
\begin{equation} \label{eq:gh-def}
 	g(z) := q^{-\frac 1{60}} G(q) \quad \text{ and } \quad h(z) := q^{\frac{11}{60}} H(q).
 \end{equation} 
They satisfy the transformations (see \cite[p. 207]{gordon-mcintosh-57})
\begin{align}
	g \sl_0 S &= \alpha^{-1} g + \beta^{-1} h, \label{eq:g-S} \\
	h \sl_0 S &= \beta^{-1} g - \alpha^{-1} h. \label{eq:h-S}
\end{align}

Using the completed functions, the mock theta conjectures \eqref{mtc-1}--\eqref{mtc-4} are implied by the corresponding completed versions:
\begin{align}
	\til f_0(z) &= - \til M\left(\mfrac 15, 10z\right) + \frac{\eta^2(5z)}{\eta(10z)} g(z), \label{mtc-1a} \\
	\til f_1(z) &= - \til M\left(\mfrac 25, 10z\right) + \frac{\eta^2(5z)}{\eta(10z)} h(z), \label{mtc-2a}\\
	\til F_0(z) &= \mfrac 12 \, \til M\left(\mfrac 15, 5z\right) - \frac{\eta^2(10z)}{\eta(5z)}h(2z), \label{mtc-3a}\\
	\til F_1(z) &= \mfrac 12 \, \til M\left(\mfrac 25, 5z\right) + \frac{\eta^2(10z)}{\eta(5z)}g(2z). \label{mtc-4a}
\end{align}
Motivated by \eqref{eq:vec-F-def} and \eqref{mtc-1a}--\eqref{mtc-4a}, we define the vector
\begin{equation} \label{eq:vec-G-def}
 \def\arraystretch{2}
	\bm{G}(z) :=  
	\left(  
	\begin{array}{c} 
	      - \til M\left(\mfrac 15, 10z\right) + \dfrac{\eta^2(5z)}{\eta(10z)} g(z) \\
	      - \til M\left(\mfrac 25, 10z\right) + \dfrac{\eta^2(5z)}{\eta(10z)} h(z) \\
	      \tfrac 12 \, \til M\left(\mfrac 15, \mfrac{5z}2\right) - \dfrac{\eta^2(5z)}{\eta(\frac{5z}2)} h(z)  \\
	      \tfrac 12 \, \til M\left(\mfrac 25, \mfrac{5z}2\right) + \dfrac{\eta^2(5z)}{\eta(\frac{5z}2)} g(z) \\
	      \tfrac 12 \, \zeta_{240} \, \til M\left(\mfrac 15, \mfrac{5z+5}2\right) - \zeta_{48}^{25}\dfrac{\eta^2(5z)}{\eta(\frac{5z+1}2)} h(z) \\
	      \tfrac 12 \, \zeta_{240}^{-71} \til M\left(\mfrac 25, \mfrac{5z+5}2\right) + \zeta_{48} \dfrac{\eta^2(5z)}{\eta(\frac{5z+1}2)} g(z)
	   \end{array}
	\right),
\end{equation}
where we have used $\eta(z+1)=\zeta_{24}\eta(z)$, $g(z+1)=\zeta_{60}^{-1}g(z)$, and $h(z+1)=\zeta_{60}^{11}h(z)$ to simplify the second terms of the fifth and sixth components.

To prove that $\bm F=\bm G$ we begin by showing that they transform in the same way.

\begin{proposition} \label{prop:G-transformation}
	The vector $\bm G(z)$ defined in \eqref{eq:vec-G-def} satisfies the transformations
	\begin{equation}
		\bm G \sl_{\frac 12} T = M_T \, \bm G \quad \text{ and } \quad \bm G \sl_{\frac 12} S = \zeta_8^{-1} \sqrt{\mfrac 25} \, M_S \, \bm G,
	\end{equation}
	where $M_T$ and $M_S$ are as in Proposition~\ref{prop:F-transformation}.
\end{proposition}

Before proving this proposition, we require two identities that will be indispensable in the proof.
Equivalent identities can be found in \cite[(3.8) and (3.9)]{gordon-mcintosh-57}, where they are proved using $q$-series methods.
In Section~\ref{sec:lem-proof} we provide a purely modular proof.

\begin{lemma} \label{lem:n1-identity}
	Let $\alpha$ and $\beta$ be as in \eqref{eq:al-be}. Then
	\begin{multline} \label{eq:n1-identity}
		\til N\left(\mfrac 15, z\right) + \alpha \, \til M \left(\mfrac 15, 25z\right) + \beta \, \til M \left(\mfrac 25, 25z\right) \\
		= 2 \frac{\eta^2(2z)}{\eta(z)} \left( \alpha^{-1} g(10z) + \beta^{-1} h(10z) \right)  - 2 \, \frac{\eta^2(50z)}{\eta(25z)} \left( \beta \, g(10z) - \alpha \, h(10z) \right).
	\end{multline}
\end{lemma}

Before proving Lemma~\ref{lem:n1-identity} we deduce an immediate consequence.
Note that the right-hand side of \eqref{eq:n1-identity} is holomorphic; this implies that the non-holomorphic completion terms on the left-hand side sum to zero.
By \eqref{eq:N-def}, the coefficients of $N(\frac a5,z)$ lie in $\Q(\zeta_5+\zeta_5^{-1})=\Q(\sqrt 5)$, and the automorphism $\sigma=(\sqrt 5\mapsto-\sqrt 5)$ maps $N(\frac 15,z)$ to $N(\frac 25,z)$.
By \eqref{eq:ntil-def} and the fact that
\[
	\csc\left(\mfrac{\pi a}5\right) = 
	\begin{cases}
		2\alpha^{-1} & \text{ if }a=1, \\
		2\beta^{-1} & \text{ if }a=2,
	\end{cases}
\]
it follows that the coefficients of both sides of \eqref{eq:n1-identity} lie in $\Q(\alpha)$.
The Galois group of $\Q(\alpha)$ is cyclic of order 4, generated by $\tau=(\alpha\mapsto\beta,\beta\mapsto-\alpha)$.
Since $\sqrt 5=\alpha\beta$, we have $\tau\mid_{Q(\sqrt 5)} = \sigma$.
Applying $\tau$ to Lemma~\ref{lem:n1-identity} gives the following identity.

\begin{lemma} \label{lem:n2-identity}
	Let $\alpha$ and $\beta$ be as in \eqref{eq:al-be}. Then
	\begin{multline} \label{eq:n2-identity}
		\til N\left(\mfrac 25, z\right) + \beta \, \til M \left(\mfrac 15, 25z\right) - \alpha \, \til M \left(\mfrac 25, 25z\right) \\
		= 2 \frac{\eta^2(2z)}{\eta(z)} \left( \beta^{-1} g(10z) - \alpha^{-1} h(10z) \right) + 2 \, \frac{\eta^2(50z)}{\eta(25z)} \left( \alpha \, g(10z) + \beta \, h(10z) \right).
	\end{multline}
\end{lemma}

\begin{proof}[Proof of Proposition~\ref{prop:G-transformation}]
	The transformation with respect to $T$ follows immediately from \eqref{eq:M-T}.

	Let $G_j(z)$ denote the $j$-th component of $\bm G(z)$.
	By \eqref{eq:M-S}, \eqref{eq:eta-S}, and \eqref{eq:g-S} we have
	\begin{align*}
		G_1(z) \sl_{\frac 12} S 
		&= \zeta_8^{-1} \sqrt{\mfrac 25} \, \left(-\mfrac 12 \, \til N\left(\mfrac 15,\mfrac z{10}\right) + \frac{\eta^2(\frac z5)}{\eta(\frac z{10})} \big(\alpha^{-1}g(z) + \beta^{-1}h(z) \big) \right) \\
		&= \zeta_8^{-1} \sqrt{\mfrac 25} \, \big(\alpha \, G_3(z) + \beta \, G_4(z) \big),
	\end{align*}
	where we used Lemma~\ref{lem:n1-identity} with $z$ replaced by $\frac z{10}$ in the second line.
	For $G_2$, the situation is analogous, using Lemma~\ref{lem:n2-identity}.
	For $G_3$ and $G_4$ we note that the transformations for $G_1$ and $G_2$ imply that
	\begin{align*}
		G_3 = \zeta_{8} \sqrt{\mfrac 25} \, \left(\mfrac \alpha2 \, G_1 \sl_{\frac 12} S + \mfrac \beta2 \, G_2 \sl_{\frac 12}S\right), \\
		G_4 = \zeta_{8} \sqrt{\mfrac 25} \, \left(\mfrac \beta2 \, G_1 \sl_{\frac 12} S - \mfrac \alpha2 \, G_2 \sl_{\frac 12}S\right).
	\end{align*}
	By \eqref{eq:sl} we have $f\sl_{\frac 12}S\sl_{\frac 12}S=f\sl_{\frac 12} (-I) = -i f$, and we obtain the transformations for $G_3$ and $G_4$.

	For $G_5$, we first observe that
	\[
		\pmmatrix 1502 \pmmatrix 0{-1}10 = T^3 S T^2 S \pmmatrix 1102.
	\]
	Using Theorems 3.1 and 3.2 of \cite{garvan} we compute
	\begin{align*}
		\til M\left(\mfrac 15,z\right) \sl_{\frac 12} T^3 S T^2 S 
		&= (-\zeta_{50}^{-3}\zeta_{24}^{-1})^3 \til M(1,3,z) \sl_{\frac12} ST^2S
		= \zeta_{200}^{39} \zeta_8^{-1} \til N(1,3,z) \sl_{\frac12}T^2S \\
		&= \zeta_{30}^{-7} \til N(0,3,z) \sl_{\frac 12} S
		= \zeta_{120}^{77} \, \til M(0,3,z).
	\end{align*}
	By (4.12) of \cite{garvan} this equals $\zeta_{120}^{47} \, \til N(\frac 25,z)$, so we conclude that
	\begin{equation} \label{eq:G5-1}
		\til M\left(\mfrac 15,\mfrac{5z+5}2\right) \sl_{\frac 12} S = \zeta_{120}^{47} \,\mfrac{1}{\sqrt 5} \, \til N\left(\mfrac 25,\mfrac{z/5+1}2\right).
	\end{equation}
	Similarly, we have $\eta(\frac{z+1}{2})\sl_{\frac 12} S = \zeta_{8}^{-1} \, \eta(\frac{z+1}{2})$ which, together with \eqref{eq:eta-S} and \eqref{eq:h-S} gives
	\begin{equation} \label{eq:G5-2}
		\zeta_{48}^{25}\dfrac{\eta^2(5z)}{\eta(\frac{5z+1}2)} h(z) \sl_{\frac 12} S = \zeta_{48}^{19} \,\mfrac{1}{\sqrt 5} \, \dfrac{\eta^2(\frac z5)}{\eta(\frac z{10}+\frac 12)} \left( \beta^{-1} g(z) - \alpha^{-1} h(z) \right).
	\end{equation}
	Replacing $z$ by $\frac z{10}+\frac 12$ in Lemma \ref{lem:n2-identity} and using \eqref{eq:M-T} yields
	\begin{multline} \label{eq:G5-3}
		\til N\left(\mfrac 25,\mfrac z{10}+\mfrac 12\right) + \beta \zeta_{60}^{-1} \, \til M \left(\mfrac 15, \mfrac{5z+5}{2}\right) - \alpha \zeta_{60}^{11} \, \til M \left(\mfrac 25, \mfrac{5z+5}{2}\right) \\
		=2 \frac{\eta^2(\frac z5)}{\eta(\frac z{10}+\frac 12)} \left(\beta^{-1} g(z) - \alpha^{-1} h(z)\right) - 2 \frac{\eta^2(5z)}{\eta(\frac{5z+1}2)} \left(\alpha g(z) + \beta h(z)\right).
	\end{multline}
	Putting \eqref{eq:G5-1}, \eqref{eq:G5-2}, and \eqref{eq:G5-3} together we obtain
	\begin{align*}
		\mfrac 12 \zeta_{240} \, \til M&\left(\mfrac 15,\mfrac{5z+5}2\right) - \zeta_{48}^{25}\dfrac{\eta^2(5z)}{\eta(\frac{5z+1}2)} h(z) \sl_{\frac 12} S \\
		&= \frac{1}{\sqrt 5} \zeta_{8}^{-1} \left( \zeta_{240} \, \mfrac\beta2 \, \til M\left(\mfrac 15,\mfrac{5z+5}2\right) + \zeta_{240}^{-71} \, \mfrac\alpha2 \, \til M\left(\mfrac 25,\mfrac{5z+5}2\right) - \zeta_{48}^{25} \frac{\eta^2(5z)}{\eta(\frac{5z+1}{2})} \left(\alpha \, g(z) + \beta \, h(z)\right) \right) \\
		&= \frac{1}{\sqrt 5} \zeta_{8}^{-1}\left(\beta \, G_5(z) + \alpha \, G_6(z)\right).
	\end{align*}
	The transformation for $G_6(z)$ is similarly obtained by using Lemma 3.
	\end{proof}

\section{Vector-valued modular forms and the Weil representation} \label{sec:weil-rep}

In this section we define vector-valued modular forms which transform according to the Weil representation, and we construct such a form from the components of $\bm F - \bm G$.
A good reference for this material is \cite[\S 1.1]{bruinier}.

Let $L=\Z$ be the lattice with associated bilinear form $(x,y)=-120xy$ and quadratic form $q(x)=-60x^2$. 
The dual lattice is $L'=\frac 1{120}\Z$. 
Let $\{\mathfrak e_h: \frac{h}{120}\in\frac1{120}\Z/\Z\}$ denote the standard basis for $\C[L'/L]$.
Let $\Mp_2(\R)$ denote the metaplectic two-fold cover of $\SL_2(\R)$; the elements of this group are pairs $(M,\phi)$, where $M=\pmatrix abcd\in \SL_2(\R)$ and $\phi^2(z)=cz+d$.
Let $\Mp_2(\Z)$ denote the inverse image of $\SL_2(\Z)$ under the covering map; this group is generated by
\[
	\til T := (T,1) \quad \text{ and } \quad \til S := (S,\sqrt z).
\]
The Weil representation can be defined by its action on these generators, namely
\begin{align}
	\rho_L(T,1)\mathfrak e_h &:= \zeta_{240}^{-h^2} \mathfrak e_h, \label{eq:rho-T} \\
	\rho_L(S,\sqrt z)\mathfrak e_h &:= \frac{1}{\sqrt{-120i}} \sum_{h'(120)} \zeta_{120}^{hh'} \, \mathfrak e_{h'}. \label{eq:rho-S}
\end{align}

A holomorphic function $\mathcal F:\H\to \C[L'/L]$ is a vector-valued modular form of weight $\frac 12$ and representation $\rho_L$ if
\begin{equation} \label{eq:weil-rep-transformation}
	\mathcal F(\gamma z) = \phi(z) \rho_L(\gamma,\phi) \mathcal F(z) \qquad \text{ for all } (\gamma,\phi) \in \Mp_2(\Z)
\end{equation}
and if $\mathcal F$ is holomorphic at $\infty$ (i.e. if the components of $\mathcal F$ are holomorphic at $\infty$ in the usual sense).
The following lemma shows how to construct such forms from vectors that transform as in Propositions \ref{prop:F-transformation} and \ref{prop:G-transformation}.

\begin{lemma} \label{lem:weil-rep}
	Suppose that $\bm H=(H_1,\ldots, H_6)$ satisfies
	\[
		\bm H \sl_{\frac 12} T = M_T \, \bm H \quad \text{ and } \quad \bm H \sl_{\frac 12} S = \zeta_{8}^{-1} \sqrt{\mfrac 25} \, M_S \, \bm H,
	\]
	where $M_T$ and $M_S$ are as in Proposition \ref{prop:F-transformation}, and
	define
	\begin{multline*}
	\mathcal H(z) := \sum_{\substack{0<h<60 \\ h\equiv \pm 1(10) \\ (h,60)=1}} \left(a_h H_3(z) + b_h H_5(z) \right)(\mathfrak e_h - \mathfrak e_{-h})
	- \sum_{\substack{0<h<60 \\ h\equiv \pm 2(10) \\ (h,60)=2}} H_1(z) (\mathfrak e_h - \mathfrak e_{-h}) \\
	+ \sum_{\substack{0<h<60 \\ h\equiv \pm 3(10) \\ (h,60)=1}} \left(a_h H_4(z) + b_h  H_6(z) \right)(\mathfrak e_h - \mathfrak e_{-h})
	- \sum_{\substack{0<h<60 \\ h\equiv \pm 4(10) \\ (h,60)=2}}  H_2(z) (\mathfrak e_h - \mathfrak e_{-h}),
	\end{multline*}
	where
	\[
	a_h = 
	\begin{cases}
		+1 & \text{ if }0<h<30, \\
		-1 & \text{ otherwise},
	\end{cases}
	\quad
	\text{ and }
	\quad
	b_h = 
	\begin{cases}
		+1 & \text{ if }h\equiv \pm 1, \pm 13 \pmod{60}, \\
		-1 & \text{ otherwise}.
	\end{cases}
	\]
	Then $\mathcal H(z)$ satisfies \eqref{eq:weil-rep-transformation}.
\end{lemma}

\begin{proof}
	The proof is a straightforward but tedious verification involving \eqref{eq:rho-T} and \eqref{eq:rho-S} that is best carried out with the aid of a computer algebra system; the author used {\sc mathematica}.
\end{proof}

\section{Proof of the mock theta conjectures} \label{sec:proof}

Let $\bm F$ and $\bm G$ be as in Section \ref{sec:preliminaries}.
To prove \eqref{mtc-1a}--\eqref{mtc-4a} it suffices to prove that $\bm F = \bm G$.
Let $\bm H := \bm F - \bm G$.
It is easy to see that the nonholomorphic parts of $\bm F$ and $\bm G$ agree, as do the terms in the Fourier expansion involving negative powers of $q$.
It follows that the function $\mathcal H$ defined in Lemma~\ref{lem:weil-rep} is a vector-valued modular form of weight $\frac 12$ with representation $\rho_L$.
By Theorem 5.1 of \cite{eichler-zagier}, the space of such forms is canonically isomorphic to the space $J_{1,60}$ of Jacobi forms of weight $1$ and index $60$.
By a theorem of Skoruppa \cite[Satz 6.1]{skoruppa-thesis} (see also \cite[Theorem 5.7]{eichler-zagier}), we have $J_{1,m}=\{0\}$ for all $m$; therefore $\mathcal H=0$.
The mock theta conjectures \eqref{mtc-1}--\eqref{mtc-4} follow. 
\qed

\section{The six remaining identities} \label{sec:remaining}

Four of the six remaining identities, those involving the mock theta functions $\psi_0$, $\psi_1$, $\phi_0$, and $\phi_1$ (see \cite{andrews-garvan} for definitions), can be proved using the methods of Sections~\ref{sec:preliminaries}--\ref{sec:proof}.
For suitable completed nonholomorphic functions $\til\psi_0$, $\til\psi_1$, $\til \phi_0$, and $\til\phi_1$, these conjectures are (see \cite[p.~206]{gordon-mcintosh-57})
\begin{align}
	2 \, \til \psi_0(z) &= \til M\left(\mfrac 15,10z\right) + 2 \, \eta_{10,1}(z) \eta(10z) h(z), \label{eq:mtc2-1} \\
	2 \, \til \psi_1(z) &= \til M\left(\mfrac 25,10z\right) + 2 \, \eta_{10,3}(z) \eta(10z) g(z), \label{eq:mtc2-2}\\
	\til \phi_0(z) &= -\mfrac 12 \, \til M\left(\mfrac 15,5z\right) + \frac{\eta(5z)\eta(2z)}{\eta(10z)} g^2(2z) h(z), \label{eq:mtc2-3} \\
	-\til \phi_1(z) &= -\mfrac 12 \, \til M\left(\mfrac 25,5z\right) + \frac{\eta(5z)\eta(2z)}{\eta(10z)} h^2(2z) g(z), \label{eq:mtc2-4}
\end{align}
where $\eta_{r,t}(z)$ is defined in \cite{robins}.
Here we have used that $g(z)h(z)=\frac{\eta(5z)}{\eta(z)}$ in the third and fourth formulas.
Following Section~\ref{sec:preliminaries}, we construct two six-dimensional vectors $\bm F_1$ and $\bm G_1$ out of the functions on the left-hand and right-hand sides, respectively, of \eqref{eq:mtc2-1}--\eqref{eq:mtc2-4}.
The transformation properties of $\bm F_1$ are given in Proposition~4.13 of \cite{zwegers}, and the corresponding properties of $\bm G_1$ follow from an argument similar to that given in the proof of Proposition~\ref{prop:G-transformation}.
For the latter argument, we use the following identity (together with the identity obtained by applying the automorphism $\tau$ as in Lemma~\ref{lem:n2-identity}):
\begin{align}
	\til N\left(\mfrac 15, 2z\right) &+ \alpha \, \til M \left(\mfrac 15, 50z\right) + \beta \, \til M \left(\mfrac 25, 50z\right) \notag \\*
	&= 2 \, \frac{\eta(2z)\eta(5z)}{\eta(z)} \left(\alpha^{-1} g(5z) + \beta^{-1} h(5z)\right)^2 \big( \alpha \, g(10z) - \beta \, h(10z) \big) \notag \\*
	& \hspace{1.5in}-2 \eta(50z) \big( \alpha \, \eta_{10,1}(5z) h(5z) + \beta \, \eta_{10,3}(5z) g(5z) \big). \label{eq:n1-id-2} 
\end{align}
The proof of \eqref{eq:n1-id-2} is analogous to the proof of Lemma~\ref{lem:n1-identity}, and requires the transformation properties of $\eta_{10,1}$ and $\eta_{10,3}$, given in \cite{robins}.
The proof that $\bm F_1 = \bm G_1$ follows exactly as in Section~\ref{sec:proof}.

The two remaining identities involve the mock theta functions $\chi_0$ and $\chi_1$.
Using the relations (discovered by Ramanujan and proved by Watson \cite[($B_0$) and ($B_1$)]{watson})
\begin{align*}
	\chi_0(q) &= 2F_0(q) - \phi_0(-q), \\
	\chi_1(q) &= 2F_1(q) + q^{-1} \phi_1(-q),
\end{align*}
these mock theta conjectures (see \cite[p. 206]{gordon-mcintosh-57}) are implied by the identities
\begin{align}
	2 \til F_0(z) - \til \phi_0(z) &= \mfrac 32 \, \til M \left(\mfrac 15,5z\right) - \eta(5z)\frac{g^2(z)}{h(z)}, \label{eq:mtc-chi-0} \\
	2 \til F_1(z) + \til \phi_1(z) &= \mfrac 32 \, \til M \left(\mfrac 25,5z\right) + \eta(5z)\frac{h^2(z)}{g(z)}. \label{eq:mtc-chi-1}
\end{align}
By \eqref{mtc-3a}, \eqref{mtc-4a}, \eqref{eq:mtc2-3}, and \eqref{eq:mtc2-4}, equations \eqref{eq:mtc-chi-0} and \eqref{eq:mtc-chi-1} follow from the identities (see \cite[(1.25) and (1.26)]{robins-thesis} for a proof using modular forms)
\begin{align*}
	g^2(z) h(2z) - h^2(z)g(2z) &= 2h(z) h^2(2z) \frac{\eta^2(10z)}{\eta^2(5z)}, \\
	g^2(z) h(2z) + h^2(z)g(2z) &= 2g(z) g^2(2z) \frac{\eta^2(10z)}{\eta^2(5z)}.
\end{align*}
This completes the proof of the remaining mock theta conjectures.

\section{Proof of Lemma~\ref{lem:n1-identity}} \label{sec:lem-proof}
Let $L(z)$ and $R(z)$ denote the left-hand and right-hand sides of \eqref{eq:n1-identity}, respectively.
Let $\Gamma$ denote the congruence subgroup
\[
	\Gamma = \Gamma_0(50) \cap \Gamma_1(5) = \left\{\pmmatrix abcd : c\equiv 0\bmod 50 \text{ and }a,d\equiv 1\bmod 5\right\}.
\]
We claim that
\begin{equation} \label{eq:claim}
	\eta(z) L(z), \eta(z) R(z) \in M_1 (\Gamma),
\end{equation}
where $M_k(G)$ (resp. $M_k^!(G)$) denotes the space of holomorphic (resp. weakly holomorphic) modular forms of weight $k$ on $G \subseteq \SL_2(\Z)$.
We have
\[
	\mfrac{1}{12}[\SL_2(\Z):\Gamma] = 15,
\]
so once \eqref{eq:claim} is established it suffices to check that the first 16 coefficients of $\eta(z) L(z)$ and $\eta(z) R(z)$ agree.
A computation shows that the Fourier expansion of each function begins
\begin{multline*}
\mfrac 2{\sqrt 5} \left( \beta - \alpha^2 \beta \, q^2 - \alpha \beta^2 \, q^3 + \beta^3 \, q^5 - \alpha^2 \beta \, q^7 + 2\alpha^2 \beta \, q^{10} - \alpha^2 \beta \, q^{12} - \alpha \beta^2 \, q^{13} + 2\alpha\beta^2 \, q^{15} + \ldots \right).
\end{multline*}

To prove \eqref{eq:claim}, we first note that Theorem 5.1 of \cite{garvan} shows that $\eta(25z)L(z) \in M_1^!(\Gamma_0(25) \cap \Gamma_1(5))$;
since $\eta(z)/\eta(25z) \in M_0^!(\Gamma_0(25))$ it follows that $\eta(z) L(z) \in M_1^!(\Gamma)$.
Suppose that $\gamma=\pmatrix abcd\in \Gamma$, and let
\[
	\gamma_n := \pmmatrix{a}{nb}{c/n}{d}.
\]
Then $\pmatrix{n}001\gamma = \gamma_{n} \pmatrix{n}001$.
By a result of Biagioli \cite[Proposition 2.5]{biagioli} we have
\begin{equation}
	g(10z) \sl_{0} \gamma = v_{\eta}^{14}(\gamma_{10}) \, g(10z) \quad \text{ and } \quad h(10z) \sl_{0} \gamma = v_{\eta}^{14}(\gamma_{10}) \, h(10z)
\end{equation}
where $v_{\eta}$ is the multiplier system for $\eta(z)$ (see \cite[(2.5)]{biagioli}).
For $d$ odd we have
\begin{equation} \label{eq:eta-mult}
	v_\eta^2(\gamma) = (-1)^{\frac{d-1}{2}} \zeta_{12}^{-ac(d^2-1)+d(b-c)}.
\end{equation}
We have
$\eta^2(2z) \sl_{1} \gamma = v_\eta^2(\gamma_2) \eta^2(2z)$
and
$\eta^2(50z) \sl_{1} \gamma = v_\eta^2(\gamma_{50}) \eta^2(50z)$.
A computation involving \eqref{eq:eta-mult} shows that
\[
	v_\eta^2(\gamma_2) v_\eta^{14}(\gamma_{10}) = v_\eta^2(\gamma_{50}) v_\eta^{14}(\gamma_{10}) = 1.
\]
It follows that $\eta(z)R(z) \in M_1^!(\Gamma)$.

It remains to show that $\eta(z)L(z)$ and $\eta(z)R(z)$ are holomorphic at the cusps.
Using {\sc magma} we compute a set of $\Gamma$-inequivalent cusp representatives:
\begin{equation} \label{eq:cusps}
\left\{\infty, \, 0, \, \mfrac{1}{8}, \, \mfrac{2}{15}, \, \mfrac{1}{7}, \, \mfrac{3}{20}, \, \mfrac{1}{6}, \, \mfrac{1}{5}, \, \mfrac{13}{50}, \, \mfrac{4}{15}, \, \mfrac{11}{40}, \, \mfrac{7}{25}, \, \mfrac{3}{10}, \, \mfrac{7}{20}, \, \mfrac{9}{25}, \, \mfrac{11}{30}, \, \mfrac{2}{5}, \, \mfrac{8}{15}, \, \mfrac{11}{20}, \, \mfrac{3}{5}, \, \mfrac{7}{10}, \, \mfrac{29}{40}, \, \mfrac{11}{15}, \, \mfrac{4}{5}\right\}.
\end{equation}
Given a cusp $\mathfrak a \in \P^1(\Q)$ and a meromorphic modular form $f$ of weight $k$ with Fourier expansion $f(z)=\sum_{n\in \Q}a(n) q^n$, the invariant order of $f$ at $\mathfrak a$ is defined as
\begin{align*}
	\ord(f,\infty) &:= \min \{n: a(n)\neq 0\}, \\
	\ord(f,\mathfrak a) &:= \ord(f\sl_{k} \gamma_{\mathfrak a}, \infty),
\end{align*}
where $\gamma_{\mathfrak a}\in \SL_2(\Z)$ sends $\infty$ to $\mathfrak a$.
For $N\in \N$, we have the relation (see e.g. \cite[(1.7)]{biagioli})
\begin{equation} \label{eq:ord-relation}
	\ord(f(Nz), \tfrac rs) = \mfrac{(N,s)^2}{N} \ord(f,\tfrac {Nr}s).
\end{equation}
We extend this definition to functions $f$ in the set
\[
	S := \{\til M(\tfrac a5,z), \til N(\tfrac a5,z) : a=1,2\} \cup \{\til M(a,b,z), \til N(a,b,z) : 0\leq a\leq 4, 1\leq b\leq 4\}
\] 
by defining the orders of these functions at $\infty$ to be the orders of their holomorphic parts at $\infty$ (see Section 6.2 and (2.1)--(2.4) of \cite{garvan}); that is,
\begin{align}
	\ord\left(\til M\left(\mfrac a5, z\right), \infty\right) =  \ord\left(\til M(a,b,z), \infty\right) &:= \mfrac{3a}{10}\left(1-\mfrac a5\right) - \mfrac 1{24}, \label{eq:m-ord}\\
	\ord\left(\til N\left(\mfrac a5, z\right), \infty\right) &:= -\mfrac 1{24}, \label{eq:n-ord-1}\\
	\ord\left(\til N(a,b,z), \infty\right) &:= \mfrac b5 k(b,5) - \mfrac{3b^2}{50} - \mfrac 1{24},\label{eq:n-ord-2}
\end{align}
where $k(b,5)=1$ if $b\in\{1,2\}$ and $2$ if $b\in\{3,4\}$.
Lastly, for $f\in S$ we define
\begin{equation}
	\ord\left(f,\mathfrak a\right) := \ord(f\sl_{\frac 12}\gamma_{\mathfrak a}, \infty).
\end{equation}
This is well-defined since $S$ is closed (up to multiplication by roots of unity) under the action of $\SL_2(\Z)$.
By this same fact we have
\begin{equation} \label{eq:eta-n-cusps}
	\min_{\text{cusps }\mathfrak a} \ord\left(\til N(\tfrac a5,z),\mathfrak a\right) \geq \min_{f\in S} \ord(f,\infty) = -\mfrac 1{24},
\end{equation}
from which it follows that
\[
	\ord\left(\eta(z)\til N(\tfrac a5,z),\mathfrak a\right) \geq 0 \qquad \text{ for all }\mathfrak a.
\]

To determine the order of $\eta(z)\til M(\tfrac a5,25z)$ at the cusps of $\Gamma$, we write
\[
	\eta(z)\til M(\tfrac a5,25z) = \frac{\eta(z)}{\eta(25z)} m(25z), \qquad \text{ where } m(z) = \eta(z)\til M(\tfrac a5,z).
\]
The cusps of $\Gamma_0(25)$ are $\infty$ and $\frac r5$, $0\leq r\leq 4$. 
By \eqref{eq:ord-relation} the function $\eta(z)/\eta(25z)$ is holomorphic at every cusp except for those which are $\Gamma_0(25)$-equivalent to $\infty$ (the latter are $\frac{13}{50}$, $\frac 7{25}$, and $\frac{9}{25}$ in \eqref{eq:cusps}); there we have $\ord(\eta(z)/\eta(25z),\infty)=-1$.
Since $\til M(\frac a5,z)$ also satisfies \eqref{eq:eta-n-cusps}, it suffices to check 
$\frac {13}{50}$, $\frac {7}{25}$, and $\frac {9}{25}$.
By \eqref{eq:ord-relation}, \cite[Theorems 3.1 and 3.2]{garvan}, the fact that $\pmatrix{13}{6}{2}{1} = T^6S^{-1}T^{-2}S$, and \eqref{eq:sl}, we have 
\begin{align*}
	\ord\left(m(25z),\mfrac{13}{50}\right) 
	&= 25 \ord\left(m(z),\mfrac{13}{2}\right) \\
	&= 25 \left(\mfrac 1{24} + \ord\left(\til M(3a \bmod 5,a,z),\infty\right)\right) = 
	\begin{cases}
		9 & \text{ if }a=1, \\
		6 & \text{ if }a=2.
	\end{cases}
\end{align*}
A similar computation shows that $\ord(m(25z),\frac 7{25}),\ord(m(25z),\frac 9{25}) \geq 4$.
Since $L(z)$ is holomorphic on $\H$,
we have, for each cusp $\mathfrak a$, the inequality
\[
	\ord(\eta(z) L(z), \mathfrak a) \geq \min \left\{ \ord\left(\eta(z) f(z),\mathfrak a\right) : f(z)=\til N(\tfrac 15,z), \til M(\tfrac 15,25z), \til M(\tfrac 25,25z) \right\} \geq 0.
\]

We turn to $\eta(z)R(z)$.
For this we require Theorem 3.3 of \cite{biagioli}, which states that
\begin{align}
	\ord(g,\tfrac rs) &= 
	\begin{cases}
		\mfrac{11}{60} & \text{ if } 5\mid s\text{ and } r\equiv \pm 2 \pmod 5, \\
		-\mfrac{1}{60} & \text{ otherwise,}
	\end{cases} \label{eq:g-cusp} \\
	\ord(h,\tfrac rs) &= 
	\begin{cases}
		\mfrac{11}{60} & \text{ if } 5\mid s\text{ and } r\equiv \pm 1 \pmod 5, \\
		-\mfrac{1}{60} & \text{ otherwise.}
	\end{cases} \label{eq:h-cusp}
\end{align}
Here we have corrected a typo in (3.2) of \cite{biagioli} (see (2.9) and Lemma 3.2 of that paper).
By \eqref{eq:ord-relation}, \eqref{eq:g-cusp}, and \eqref{eq:h-cusp} we have
\[
	\ord\left( \eta(z)R(z), \tfrac rs \right) \geq - \mfrac{(10,s)^2}{600} + \min \left\{ \mfrac 1{24} + \mfrac{(50,s)^2-(25,s)^2}{600}, \mfrac{(2,s)^2}{24} \right\}.
\]
Since the latter expression is nonnegative for all $s\mid 50$, it follows that $\eta(z)R(z)$ is holomorphic at the cusps.
This completes the proof.
\qed

\bibliographystyle{alphanum-2}
\bibliography{mtc-bib}

\end{document}